\documentclass[10pt]{amsart}
\usepackage{amsfonts}
\usepackage{amsmath}
\usepackage{amsthm,amscd}
\usepackage{tikz-cd}
\usepackage{mathtools}
\usepackage{amssymb}
\usepackage{mathrsfs}
\usepackage{enumerate}
\usepackage[mathscr]{euscript}
\usepackage{etex}
\usepackage{hyperref}
\usepackage{svg}
\hypersetup{
    colorlinks=true,
    linkcolor=blue,
    filecolor=red,      
    urlcolor=teal,
    citecolor=teal,
}
\usepackage{color}

\usepackage{tikz}
\usepackage{xypic}
\usepackage{longtable}

\usepackage[T1]{fontenc}
\usepackage[utf8]{inputenc}

\usepackage{graphicx}

\allowdisplaybreaks

\calclayout

\theoremstyle{plain}\newtheorem{Theorem}{Theorem}[section]
\theoremstyle{plain}\newtheorem{Corollary}[Theorem]{Corollary}
\theoremstyle{plain}\newtheorem{Lemma}[Theorem]{Lemma}
\theoremstyle{plain}\newtheorem{Definition}[Theorem]{Definition}
\theoremstyle{plain}
\theoremstyle{plain}\newtheorem{Proposition}[Theorem]{Proposition}
\theoremstyle{plain}
\theoremstyle{plain}
\theoremstyle{plain}
\theoremstyle{definition}\newtheorem{Remark}[Theorem]{Remark}
\theoremstyle{plain}\newtheorem*{Theorem*}{Theorem}
\theoremstyle{plain}\newtheorem*{Problem*}{Problem}
\theoremstyle{plain}\newtheorem*{Principle*}{Principle}

\theoremstyle{plain}

\makeatletter
\newtheorem*{rep@theorem}{\rep@title}
\newcommand{\newreptheorem}[2]{%
\newenvironment{rep#1}[1]{%
 \def\rep@title{#2 \ref{##1}}%
 \begin{rep@theorem}}%
 {\end{rep@theorem}}}
\makeatother
\theoremstyle{plain}\newreptheorem{theorem}{Theorem}

\theoremstyle{remark}
\theoremstyle{remark}\newtheorem{notation}[Theorem]{Notation}
\theoremstyle{remark}

\usepackage{amsmath,float}
\usetikzlibrary{decorations.markings}

\usepackage{graphicx}
\usepackage{subcaption}
\numberwithin{equation}{section}
\newcommand{\bA}{\mathbb{A}}

\newcommand{\bP}{\mathbb{P}}

\newcommand{\bZ}{\mathbb{Z}}

\newcommand{\clM}{\mathcal{M}}
\newcommand{\clN}{\mathcal{N}}
\newcommand{\clO}{\mathcal{O}}

\newcommand{\fm}{\mathfrak{m}}

\newcommand{\fF}{\mathfrak{F}}

\newcommand{\fS}{\mathfrak{S}}

\newcommand{\sC}{\mathscr{C}}
\newcommand{\sD}{\mathscr{D}}
\newcommand{\sE}{\mathscr{E}}

\newcommand{\sI}{\mathscr{I}}

\newcommand{\sL}{\mathscr{L}}

\newcommand{\sS}{\mathscr{S}}

\newcommand{\sX}{\mathscr{X}}
\newcommand{\sY}{\mathscr{Y}}

\newcommand{\Mbar}{\overline{\mathcal{M}}}
\DeclareMathOperator\Spec{Spec}
\DeclareMathOperator\Proj{Proj}

\DeclareMathOperator\Hom{Hom}

\DeclareMathOperator\rank{rank}

\begin{document}

\author{Fatemeh Rezaee}
\address{University of Cambridge}
\email{fr414@cam.ac.uk}
\author{Mohan Swaminathan}
\address{Tata Institute of Fundamental Research, Mumbai}
\email{mohans@math.tifr.res.in}

\title[An obstruction to smoothing stable maps]{An obstruction to smoothing stable maps}
\begin{abstract}
    We describe an obstruction to smoothing stable maps in smooth projective varieties, which generalizes some previously known obstructions. Our obstruction comes from the non-existence of certain rational functions on the ghost components, with prescribed simple poles and residues. We illustrate our result by giving examples of non-smoothable stable maps which are not detected by previous works.
\end{abstract}

\maketitle

\setcounter{tocdepth}{1}

\section{Introduction}

The main result of the present article is an obstruction to smoothing stable maps. This is a sequel to our previous work \cite{prequel}, where we gave a sufficient condition for smoothing stable maps, and is the next step towards the construction of a natural compactification of the space of smooth curves mapping to a variety which is smaller than the moduli space of stable maps.
We work over a fixed algebraically closed field $k$ of characteristic $0$ and let $X$ denote a smooth projective variety. We start by recalling some basic definitions.

\begin{Definition}[Prestable curve, stable map]\label{def:Prestable curve, stable map}
    A \emph{prestable curve} $\sC$ is a projective, connected, reduced curve with at worst ordinary nodal singularities. A \emph{stable map} is a morphism $f:\sC\to X$ from a prestable curve $\sC$ such that there are only finitely many automorphisms of $\sC$ which preserve the morphism $f$.
\end{Definition}

\begin{Definition}[Ghost, effective sub-curve]
    For a stable map $f:\sC\to X$, the \emph{ghost sub-curve} is the union of the irreducible components of $\sC$ on which $f$ is constant. Each connected component $C$ of the ghost sub-curve is called a \emph{ghost component} of $(\sC,f)$. The \emph{effective sub-curve} of $(\sC,f)$ is the closure of the complement of the ghost sub-curve.
\end{Definition}

\begin{Definition}[Eventual smoothability]
    Denote the moduli stack of stable maps to $X$, of all genera and degrees, by $\Mbar(X)$. Let $\clM(X)$ denote the open substack of $\Mbar(X)$ which parametrizes maps with non-singular domain curves. 
    
    We say that a stable map $f:\sC\to X$ is \emph{smoothable} if $(\sC,f)$, as a  point of $\Mbar(X)$, lies in the closure of $\clM(X)$. We say that $f:\sC\to X$ is \emph{eventually smoothable} if there exists an embedding $\iota:X\hookrightarrow X'$ into another smooth projective variety such that the map $\iota\circ f:\sC\to X'$ is smoothable.
\end{Definition}

\begin{Remark}
    In the definition of eventual smoothability, observe that it suffices to let $X'$ range over all projective spaces $\bP^N$, rather than all smooth projective varieties. This recovers the original definition given in \cite[Definition~1.2]{prequel}.
\end{Remark}

Smoothability implies eventual smoothability, but the converse is false in general. In \cite[Theorem~B]{prequel}, we identified a large class of stable maps which we call `stable maps with model ghosts' and showed that these are eventually smoothable.

\begin{Definition}[Residue obstruction, evaluation map]\label{def:residue-obstruction}
    Let $p$ be a non-singular closed point on a prestable curve $C$. The connecting map in the cohomology exact sequence associated to the short exact sequence $0\to\clO_C\to\clO_C(p)\to T_{C,p}\to 0$ is denoted by 
    \begin{align*}
        \delta_{C,p}:T_{C,p}\to H^1(C,\clO_C).
    \end{align*}
    Here, by a slight abuse of notation, we denote the skyscraper sheaf on $C$ supported at $p$ with stalk $T_{C,p}$ by the same symbol as for the stalk. The map $\delta_{C,p}$ corresponds under Serre duality to the map
    \begin{align*}
        \normalfont\text{ev}_{C,p}:H^0(C,\omega_C)\to T_{C,p}^\vee
    \end{align*}
    given by evaluation at $p$. Here, $\omega_C$ is the dualizing line bundle of $C$.
\end{Definition}
\begin{Definition}[First derivative]    
   For a stable map $f:\sC\to X$, and a non-singular closed point $p\in\sC$, we denote the derivative of $f$ at $p$ by
   \begin{align*}
       df_{p}: T_{\sC,p}\to T_{X,f(p)}.
   \end{align*}
\end{Definition}

We are now in a position to state our main result.

\begin{Theorem}[Obstruction to eventual smoothability]\label{thm:first-order-obstruction}
    Let $X$ be a smooth projective variety and let $f:\sC\to X$ be a non-constant stable map. 
    Let $C$ be a ghost component of $(\sC,f)$ mapping to a point $q\in X$ and let $\sE$ be the effective sub-curve. If $(\sC,f)$ is eventually smoothable, then the linear map
    \begin{align}\label{eqn:first-order-obstruction-map}
        \bigoplus_{p\in C\cap\sE}\delta_{C,p}\otimes d(f|_\sE)_{p}:\bigoplus_{p\in C\cap\sE} T_{C,p}\otimes T_{\sE,p}\to H^1(C,\clO_C)\otimes T_{X,q}
    \end{align}
    has a non-trivial kernel.
\end{Theorem}

To obstruct eventual smoothability of a stable map $(\sC,f)$ using this theorem, one needs to analyze the map \eqref{eqn:first-order-obstruction-map}, which involves interactions between $\delta_{C,p}$ and $d(f|_\sE)_p$ for $p\in C\cap\sE$. The following corollary is easier to use as it sidesteps these interactions.

\begin{Corollary}\label{cor:rank-obstruction}
    Let $X$ be a smooth projective variety and let $f:\sC\to X$ be a non-constant eventually smoothable stable map with effective sub-curve $\sE$. If $C$ is any ghost component of $(\sC,f)$, mapping to $q\in X$, then there is a non-empty subset $D\subset C\cap\sE$ such that
    \begin{align}\label{eqn:rank-obstruction}
        \rank\left[\bigoplus_{p\in D}T_{\sE,p}\to T_{X,q}\right] + \rank\left[ H^0(C,\omega_C)\to\bigoplus_{p\in D}T_{C,p}^\vee\right]\le|D|.
    \end{align}
\end{Corollary}
\begin{proof}
    Use Theorem~\ref{thm:first-order-obstruction} to find a nonzero element $x = (x_p)_{p\in C\cap\sE}$ in the kernel of \eqref{eqn:first-order-obstruction-map}. Define $D\subset C\cap\sE$ to consist of those $p$ for which $x_p$ is nonzero and therefore provides an isomorphism $x_p:T_{C,p}^\vee\to T_{\sE,p}$. Since \eqref{eqn:first-order-obstruction-map} maps $x$ to zero, the composition
    \begin{align*}
        H^0(C,\omega_C)\to\bigoplus_{p\in D}T_{C,p}^\vee\xrightarrow{\sim}\bigoplus_{p\in D}T_{\sE,p}\to T_{X,q}
    \end{align*}
    must be zero. The result now follows from the rank-nullity theorem.
\end{proof}

\begin{Remark}
    The class of examples discussed in \textsection\ref{subsubsec:example-1} demonstrate that the obstruction to eventual smoothability provided by Corollary~\ref{cor:rank-obstruction} is strictly weaker than the one coming from Theorem~\ref{thm:first-order-obstruction}. See Remark~\ref{rem:thm-stronger-than-corollary} for more details.
\end{Remark}

Theorem~\ref{thm:first-order-obstruction} and Corollary~\ref{cor:rank-obstruction} provide obstructions to eventual smoothability which are insensitive to the global geometry of $X$. Rather, they are sensitive only to the behavior of the stable map near its ghost components. This is consistent with the {`local criterion for eventual smoothability'} \cite[Theorem~A]{prequel}.

\subsection*{Outline of proof} We briefly explain the idea behind the proof of Theorem \ref{thm:first-order-obstruction}. If the stable map $f:\sC\to X$ is smoothable, then we obtain a family of stable maps 
\begin{center}
\begin{tikzcd}
    \sS \arrow[d] \arrow[r,"F"] & X \\
    \Delta &
\end{tikzcd}
\end{center}
over a disk $\Delta$ with coordinate $t$, with the following property. Its fibre over $t=0$ is identified with $f:\sC\to X$, and its general fibre is a smooth curve mapping to $X$. We expand $F:\sS\to X$ near the ghost component $C\subset\sS$ as a power series $F = F_0 + F_1t + F_2t^2 + \cdots$, in local coordinates centred at the point $q = f(C)\in X$. 

Near each $p\in C\cap\sE$, the surface $\sS$ is given by $xy=t^{m_p}$, for some $m_p\ge 1$. Define $m = \min_{p\in C\cap\sE} m_p$, and let $Z\subset C\cap\sE$ be the set where this minimum is attained. The key observation is that the rational function $F_m:C\dashrightarrow T_{X,q}$ is regular away from $Z$, and at each $p\in Z$, it has (at worst) a simple pole with residue proportional to $d(f|_\sE)_p$. Thus, the obstruction to the existence of this rational function with prescribed simple poles and residues, which lies in $H^1(C,\clO_C)\otimes T_{X,q}$, must vanish.

\subsection*{Relations to previous work} The question of eventual smoothability of stable maps has been studied in many earlier works in algebraic and symplectic geometry, albeit using different terminology. We give a brief overview here and refer the reader to \cite[\textsection 1.3]{prequel} for a more thorough discussion.

Algebraic geometry: The genus $1$ case of Theorem~\ref{thm:first-order-obstruction} was obtained in \cite[Lemma~5.9]{Vakil-genus01}. When $X$ is a surface, obstructions to smoothing for general genus are given \cite[Lemma~2.4.1]{Vakil-stable-red} by comparing the stable map limit of a family of curves with the limit in the Hilbert scheme. It should be possible to derive obstructions for general genus and $X$ of general dimension, using the compactification of $\clM(X)$ by \emph{stable unramified maps}, constructed in \cite{KKO}; this is because this compactification has a natural non-surjective forgetful map to $\Mbar(X)$. Additionally, smooth modular compactifications of $\clM_{g,n}(\bP^N,d)$ are known for $g = 1$ \cite{RSPW19} and $g = 2$ \cite{BC23}; these are closely related to their earlier counterparts in the more general symplectic setting mentioned in the next paragraph.

Symplectic geometry: Optimal obstructions to eventual smoothability of stable maps are known in genus $1$ \cite{Zinger-sharp-compactness} and genus $2$ \cite{Niu-thesis}. Comparing with these two works, we find that Theorem~\ref{thm:first-order-obstruction} recovers the optimal obstruction for genus $1$ but not for genus $2$ (cf. Remark~\ref{rem:dim-count-vs-obstruction}). For the situation where a ghost component of general genus meets the effective sub-curve in at most two points, obstructions to eventual smoothability were obtained in \cite{DW} and \cite{ekholm-shende-ghost}. For algebraic target manifolds, Corollary~\ref{cor:rank-obstruction} extends the results of these two works. We point out that our map \eqref{eqn:first-order-obstruction-map} is the algebro-geometric analogue of the `leading order term of the obstruction on ghost components' from \cite[Lemma~5.47]{DW}.

\subsection*{Structure of the paper}
In section~\ref{sec:examples}, we give examples of stable maps whose eventual smoothability can be ruled out using Theorem~\ref{thm:first-order-obstruction}. The examples we give demonstrate that Theorem~\ref{thm:first-order-obstruction} does not follow from Corollary~\ref{cor:rank-obstruction} or the results of \cite{Vakil-genus01, Vakil-stable-red, Zinger-sharp-compactness, Niu-thesis, DW, ekholm-shende-ghost}. In section~\ref{sec:proof}, we prove Theorem~\ref{thm:first-order-obstruction}.

\subsection*{Acknowledgements}

We would like to thank Ben Church, Mark Gross, Eleny Ionel, Eric Kilgore, Melissa Liu, Rahul Pandharipande, John Pardon, Dhruv Ranganathan, Richard Thomas and Ravi Vakil for useful conversations and encouragement. We also thank the anonymous referee for their careful reading of the paper and for their useful comments and suggestions. F.R. was supported by UKRI grant No. EP/X032779/1, and the ERC Advanced Grant MSAG. 

\subsection*{Conventions} All schemes we consider will be quasi-projective over the fixed algebraically closed field $k$ of characteristic $0$. The word \emph{variety} is reserved for reduced and irreducible schemes. The notation $\bA^n$ (resp. $\bP^n$) denotes the affine (resp. projective) $n$-space over $k$. Tensor products and dimensions written without subscripts are understood to be over $k$.

For a scheme $S$ and a point $p\in S$, its local ring at $p$ is denoted by $\clO_{S,p}$ and its maximal ideal is denoted by $\fm_{S,p}$. An \emph{\'etale neighborhood} of $p$ in $S$ is a scheme $U$ with an \'etale morphism $U\to S$ under which the inverse image of $p$ is a single point. 

When $p$ is a closed point of the scheme $S$, we denote the Zariski tangent space of $S$ at $p$ by $T_{S,p} = ({\fm_{S,p}}/{\fm_{S,p}^2})^\vee$. When $S$ is smooth of dimension $n$ at the closed point $p$, then elements $z_1,\ldots,z_n\in\clO_{S,p}$ are called \emph{local coordinates} for $S$ at $p$ if they generate the maximal ideal $\fm_{S,p}$. The images $dz_1,\ldots,dz_n$ of the local coordinates $z_1,\ldots,z_n$ in ${\fm_{S,p}}/{\fm_{S,p}^2} = T_{S,p}^\vee$ form a basis, and we denote the corresponding dual basis of $T_{S,p}$ by $\partial_{z_1},\ldots,\partial_{z_n}$.

When $Z$ is an local complete intersection subscheme of a smooth variety $Y$ defined by the ideal sheaf $\sI$, the \emph{normal bundle} of $Z\subset Y$ is the locally free sheaf on $Z$ given by $N_{Z/Y}=({\sI}/{\sI^2})^\vee$. We often use the fact that when $D$ is a reduced effective divisor in the smooth variety $Y$, we have a natural identification $N_{D/Y}^{\otimes m} = \clO_Y(mD)|_D$ for all $m\in\bZ$. In particular, when $p\in C$ is a non-singular closed point on a curve, we have $N_{p/C} = T_{C,p}$ and therefore, we have $\clO_C(mp)|_p = T_{C,p}^{\otimes m}$ for all $m\in\bZ$.
\section{Non-smoothable stable maps to projective space}\label{sec:examples}

In this section, we give some concrete examples to illustrate Theorem~\ref{thm:first-order-obstruction}. Before discussing the examples, we use an elementary dimension counting argument to show that, in projective space, stable maps with ghost components are usually not smoothable.

\subsection{Dimension counting}\label{subsec:dim-count}

Fix integers $N,g,d\ge 1$ with $d\ge 2g-1$, and consider the moduli stack $\Mbar := \Mbar_g(\bP^N,d)$ of stable genus $g$ maps of degree $d$ to $\bP^N$. 

For any smooth projective curve $\sC$ of genus $g$, any degree $d$ morphism $f:\sC\to\bP^N$ has unobstructed deformations, since we have $H^1(\sC,f^*T_{\bP^N}) = 0$. This comes from the Euler exact sequence $0\to\clO_{\bP^N}\to\clO_{\bP^N}(1)^{\oplus(N+1)}\to T_{\bP^N}\to 0$, and 
\begin{align*}
    H^1(\sC,f^*\clO_{\bP^N}(1)) = H^0(\sC,\omega_\sC\otimes f^*\clO_{\bP^N}(-1))^\vee = 0,    
\end{align*}
using Serre duality. Thus, the open substack $\clM:=\clM_g(\bP^N,d)$ of $\Mbar$ parametrizing stable maps with non-singular domain curves is smooth and has dimension
\begin{align*}
    \dim \clM = \chi(\sC,f^*T_{\bP^N}) + \dim\clM_g = (N-3)(1-g) + d(N+1).
\end{align*}

\begin{Remark}
    For degrees $1\le d\le 2g-2$, the situation is more complicated and is related to Brill--Noether theory. The existence of special divisors in low degrees prevents us from using the argument above to conclude that the stack $\clM_g(\bP^N,d)$ is smooth of the expected dimension.
\end{Remark}

Let us now describe a stratum within $\Mbar\setminus\clM$. Let $n,h\ge 1$ be integers with $h\le g$. Let $d_1,\ldots,d_n\ge 1$ and $g_1,\ldots,g_n\ge 0$ be integers such that $d_i\ge 2g_i$, and we have $d = \sum d_i$ and $g = h + \sum g_i$. For instance, we could take $n = 1$ and $h=g$. Consider the stack
\begin{align*}
    \clN:=\clM_{h,n}\times(\clM_{g_1,1}(\bP^N,d_1)\times_{\bP^N}\cdots\times_{\bP^N}\clM_{g_n,1}(\bP^N,d_n))
\end{align*}
parametrizing the data of an $n$-pointed smooth projective genus $h$ curve $(C,p_1,\ldots,p_n)$ and, for each $1\le i\le n$, a degree $d_i$ morphism $f_i:\sE_i\to \bP^N$ from a pointed smooth projective genus $g_i$ curve $(\sE_i,p_i')$ such that $f_1(p_1') = \cdots = f_n(p_n')$. We have a natural morphism
\begin{align*}
    \iota:\clN\to\Mbar,
\end{align*}
given by mapping $(C,p_1,\ldots,p_n),(\sE_1,p_1',f_1),\ldots,(\sE_n,p_n',f_n')$ to the stable map $(\sC,f)$, where $\sC$ is the prestable curve obtained from $C\sqcup\bigsqcup \sE_i$ by making the identifications $p_i\sim p_i'$, and $f$ is the morphism which restricts to $f_i$ on $\sE_i$ and is constant on $C$. The morphism $\iota$ from $\clN$ onto its image is finite-to-one. For $1\le i\le n$, the morphism
\begin{align*}
    \clM_{g_i,1}(\bP^N,d_i)&\to\bP^N
    \\(f_i:\sE_i\to X,p_i')&\mapsto f_i(p_i')
\end{align*}
given by evaluation at the marked point is smooth since Serre duality gives
\begin{align*}
    H^1(\sE_i,\clO_{\sE_i}(-p_i')\otimes f_i^*\clO_{\bP^N}(1)) = H^0(\sE_i,\omega_{\sE_i}(p_i')\otimes f_i^*\clO_{\bP^N}(-1))^\vee = 0.    
\end{align*}
Thus, $\clN$ is smooth and has dimension
\begin{align*}
    \dim\clN &= \dim\clM_{h,n} - N(n-1) + \sum_{i=1}^{n}\big(\chi(\sE_i,f_i^*T_{\bP^N})+\dim\clM_{g_i,1}\big)
    \\
    &= 3h-3+n-N(n-1) + \sum_{i=1}^{n}\big( (N-3)(1-g_i)+d_i(N+1) +1\big) \\
    &= \dim\clM + Nh-n.
\end{align*}
When $Nh-n\ge 0$, we get $\dim\clN\ge \dim\clM$, which implies that most points of $\iota(\clN)$ are not contained in the closure of $\clM$. More precisely, the pullback of the closure of $\clM$ under the morphism $\iota$ is of codimension $\ge Nh-n+1$ in $\clN$.

\begin{Remark}\label{rem:dim-count-vs-obstruction}
    It is instructive to compare the conclusion obtained above via dimension counting with what our main result gives in this situation.
    
    Theorem~\ref{thm:first-order-obstruction} asserts that \eqref{eqn:first-order-obstruction-map}, which is a linear map from an $n$-dimensional vector space to an $Nh$-dimensional vector space, must have a non-trivial kernel at all points of $\iota(\clN)$ that lie in the closure of $\clM$. When $Nh\ge n$, having a non-trivial null-space is a codimension $Nh-n+1$ condition on the space of $Nh\times n$ matrices. Thus, the above dimension count is consistent with what our main result predicts.

    However, the same dimension count can also be used to show the existence of points in $\iota(\clN)$ whose non-smoothability  cannot be detected by our main result. Specifically, consider the case where $n = 1$ and simplify the notation by writing $\sE = \sE_1$, $p' = p'_1$ and $p = p_1$. In this situation, the map $\delta_{C,p}$ is always injective and Theorem~\ref{thm:first-order-obstruction} simply says that points of $\iota(\clN)$ lying in the closure of $\clM$ must satisfy $d(f|_\sE)_{p'} = 0$. After pulling back via $\iota$, the vanishing of $d(f|_\sE)_{p'}$ defines a locus in $\clN$ of codimension $\le N$. On the other hand, we know that the points of $\clN$ whose image under $\iota$ lies in the closure of $\clM$ form a locus of codimension $\ge Nh-n+1 = Nh$ in $\clN$. In particular, when $h>1$, we deduce the existence of non-smoothable stable maps in $\iota(\clN)$ which nevertheless satisfy $d(f|_\sE)_{p'} = 0$.
\end{Remark}

\subsection{Concrete examples}\label{subsec:concrete-examples}
Keeping the above setup, we now describe three classes of examples which illustrate Theorem \ref{thm:first-order-obstruction}.

\subsubsection{Examples in genus one} With notation as in \textsection\ref{subsec:dim-count}, take
\begin{align*}
    N,n\ge 1, \quad h=1, \quad g_1=\cdots=g_n=0 \quad\text{and}\quad d_1,\ldots,d_n\ge 1.
\end{align*}

Let $C$ be a smooth projective curve of genus $1$. Then the map
\begin{align*}
    \delta_{C,p}:T_{C,p}\to H^1(C,\clO_C).
\end{align*}
is an isomorphism for any closed point $p\in C$. Choose a collection of distinct closed points $p_i\in C$ for $1\le i\le n$. Let $q\in\bP^N$ be a closed point. For each $1\le i\le n$, let $(\sE_i,p_i')$ be a copy of $(\bP^1,0)$ and choose a morphism $f_i:\sE_i\to\bP^N$ of degree $d_i$ with the property that $f_i(p_i') = q$.

We now define a stable map $f:\sC\to\bP^N$ as follows. Define $\sC$ to be the prestable curve obtained from $C\sqcup\bigsqcup\sE_i$ by making the identifications $p_i\sim p_i'$. Define $f$ to be the constant morphism with value $q$ on $C$ and the morphism $f_i$ on $\sE_i$, for $1\le i\le n$. By construction, the map \eqref{eqn:first-order-obstruction-map} is injective if and only if the map
\begin{align}\label{eqn:genus-1-first-order-obstruction-map}
    \bigoplus_{1\le i\le n} d(f|_{\sE_i})_{p_i}:\bigoplus_{1\le i\le n} T_{\sE_i,p_i}\to T_{X,q}
\end{align}
is injective. We conclude from Theorem~\ref{thm:first-order-obstruction} that, if the stable map $f:\sC\to \bP^N$ is eventually smoothable, then the map \eqref{eqn:genus-1-first-order-obstruction-map} has a non-trivial kernel. This is exactly the content of \cite[Lemma 5.9]{Vakil-genus01} in this case. 

In fact, \cite{Zinger-sharp-compactness} shows that even more is true: the smoothability of the stable map $f:\sC\to\bP^N$ is equivalent to the map \eqref{eqn:genus-1-first-order-obstruction-map} having a non-trivial kernel.

\subsubsection{First class of examples in higher genus}\label{subsubsec:example-1} 
With notation as in \textsection\ref{subsec:dim-count}, take
\begin{align*}
    N,h\ge 2, \quad n = Nh, \quad g_1 = \cdots = g_n = 0 \quad\text{and}\quad d_1 = \cdots = d_n = 1.    
\end{align*}

Let $C$ be a smooth projective genus $h$ curve. Choose a collection of distinct closed points $p_{i,j}\in C$ for $1\le i\le N$ and $1\le j\le h$, with the following property: for each $1\le i\le N$, the points $\{p_{i,j}\}_{1\le j\le h}$ have linearly independent images under the natural morphism 
\begin{align}\label{eqn:canonical-morphism}
    \iota_C: C\to \bP H^0(C,\omega_C)^\vee,   
\end{align}
defined by the canonical bundle $\omega_C$ of $C$. Note that since $C$ is smooth, its canonical bundle is the same as its dualizing line bundle. From Serre duality, it follows that the map
\begin{align*}
    \bigoplus_{1\le j\le h}\delta_{C,p_{i,j}}:\bigoplus_{1\le j\le h}T_{C,p_{i,j}}\to H^1(C,\clO_C)
\end{align*}
is an isomorphism for $1\le i\le N$. Let $q\in\bP^N$ be a closed point and let $L_1,\ldots,L_N\subset\bP^N$ be lines passing through $q$ such that we have
\begin{align*}
    T_{\bP^N,q} = \bigoplus_{1\le i\le N} T_{L_i,q}.
\end{align*}
For each $1\le i\le N$ and $1\le j\le h$, let $(\sE_{i,j},p'_{i,j})$ be a copy of $(L_i,q)$.

We now define a stable map $f:\sC\to\bP^N$ as follows. Define $\sC$ to be the prestable curve obtained from $C\sqcup\bigsqcup \sE_{i,j}$ by making the identifications $p_{i,j}\sim p_{i,j}'$. The morphism $f$ is defined to be the constant morphism with value $q$ on $C$ and the inclusion $\sE_{i,j}\simeq L_i\subset \bP^N$ for $1\le i\le N$ and $1\le j\le h$. Figure~\ref{fig:h=4,N=3} illustrates the case where $N=3$ and $h=4$.

\begin{figure}[h]
    \centering
    \includegraphics[width=0.95\linewidth]{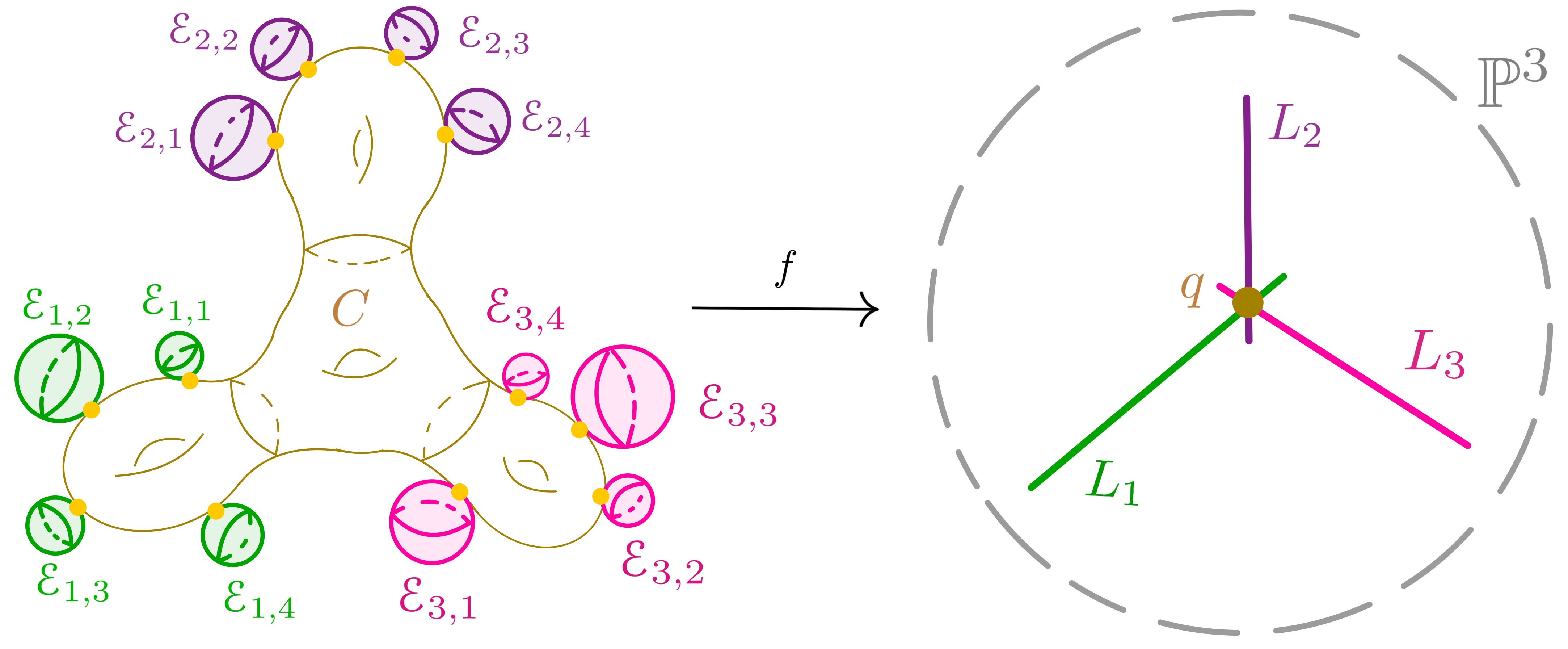}
    \caption{An example of a stable map which is not eventually smoothable, obtained by taking $N=3$, $h=4$ and $n=Nh=12$ in \textsection\ref{subsubsec:example-1}. For $1\le i\le 3$, the curves $\sE_{i,1},\ldots,\sE_{i,4}$ are attached at four general points of $C$. 
    }
    \label{fig:h=4,N=3}
\end{figure}

By construction, the map \eqref{eqn:first-order-obstruction-map} is an isomorphism. We conclude from Theorem~\ref{thm:first-order-obstruction} that the stable map $f:\sC\to\bP^N$ is not eventually smoothable.

\begin{Remark}\label{rem:thm-stronger-than-corollary}
   Since we assumed $N,h\ge 2$, we have $n = Nh\ge N+h$, implying that the inequality \eqref{eqn:rank-obstruction} holds when we take $D$ to consist of all the points $p_{i,j}$. This demonstrates that Theorem~\ref{thm:first-order-obstruction} is strictly stronger than Corollary~\ref{cor:rank-obstruction}. 
\end{Remark}

\subsubsection{Second class of examples in higher genus}

With notation as in \textsection\ref{subsec:dim-count}, take
\begin{align*}
    N = 2,\quad h\ge 2, \quad n = 2h, \quad g_1 = \cdots = g_n = 0 \quad\text{and}\quad d_1 = \cdots = d_n = 1.    
\end{align*}
Let $C$ be a smooth genus $h$ curve which is \emph{hyperelliptic}, i.e., the morphism \eqref{eqn:canonical-morphism} realizes $C$ as a branched double cover of a rational normal curve\footnote{A \emph{rational normal curve} in $\bP^m$ is a smooth rational curve of degree $m$ which is not contained in any hyperplane. In suitable coordinates, it is parametrized by $[s:t]\mapsto[s^m:s^{m-1}t:\cdots:t^m]$.} $D\subset\bP H^0(C,\omega_C)^\vee$. Let $\iota:C\to C$ be the non-trivial automorphism of this double cover. Choose a collection of distinct closed points $p_{1,j}\in C$ with $1\le j\le h$ such that none of them are fixed by $\iota$ and their images under \eqref{eqn:canonical-morphism} are linearly independent. Thus, the map
\begin{align*}
    \bigoplus_{1\le j\le h}\delta_{C,p_{1,j}}:\bigoplus_{1\le j\le h}T_{C,p_{1,j}}\to H^1(C,\clO_C)
\end{align*}
is an isomorphism. For each $1\le j\le h$, define $p_{2,j} = \iota(p_{1,j})\in C$ and note that the maps $\delta_{C,p_{i,j}}$ for $i=1,2$ have a common line $\ell_j\subset H^1(C,\clO_C)$ as their image.

Let $q\in\bP^2$ be a closed point and let $L_1,\ldots,L_h\subset\bP^2$ be a collection of distinct lines passing through $q$. For $1\le j<l\le h$, we then have
\begin{align*}
    T_{\bP^2,q} = T_{L_j,q}\oplus T_{L_l,q}.
\end{align*}
For $1\le i\le 2$ and $1\le j\le h$, let $(\sE_{i,j},p'_{i,j})$ be a copy of $(L_j,q)$. 

Let $\sigma\in\fS_h$ be a permutation of $\{1,\ldots,h\}$. Define $\sC_\sigma$ to be the prestable curve obtained from $C\sqcup\bigsqcup \sE_{i,j}$ by making the identifications
\begin{align*}
    p_{1,j}\sim p_{1,j}'\quad\text{and}\quad  p_{2,j}\sim p_{2,\sigma(j)}'.  
\end{align*}
Define the stable map $f_\sigma:\sC_\sigma\to\bP^2$ to be the constant morphism with value $q$ on $C$ and the inclusion $\sE_{i,j}\simeq L_j\subset \bP^2$ for $1\le i\le 2$ and $1\le j\le h$. Figure \ref{fig:h=3,N=2 nonsmoothable} illustrates the case where $h=3$ and $\sigma = (231)\in\fS_3$.

\begin{figure}[h]
    \centering
\includegraphics[width=1.02\linewidth]{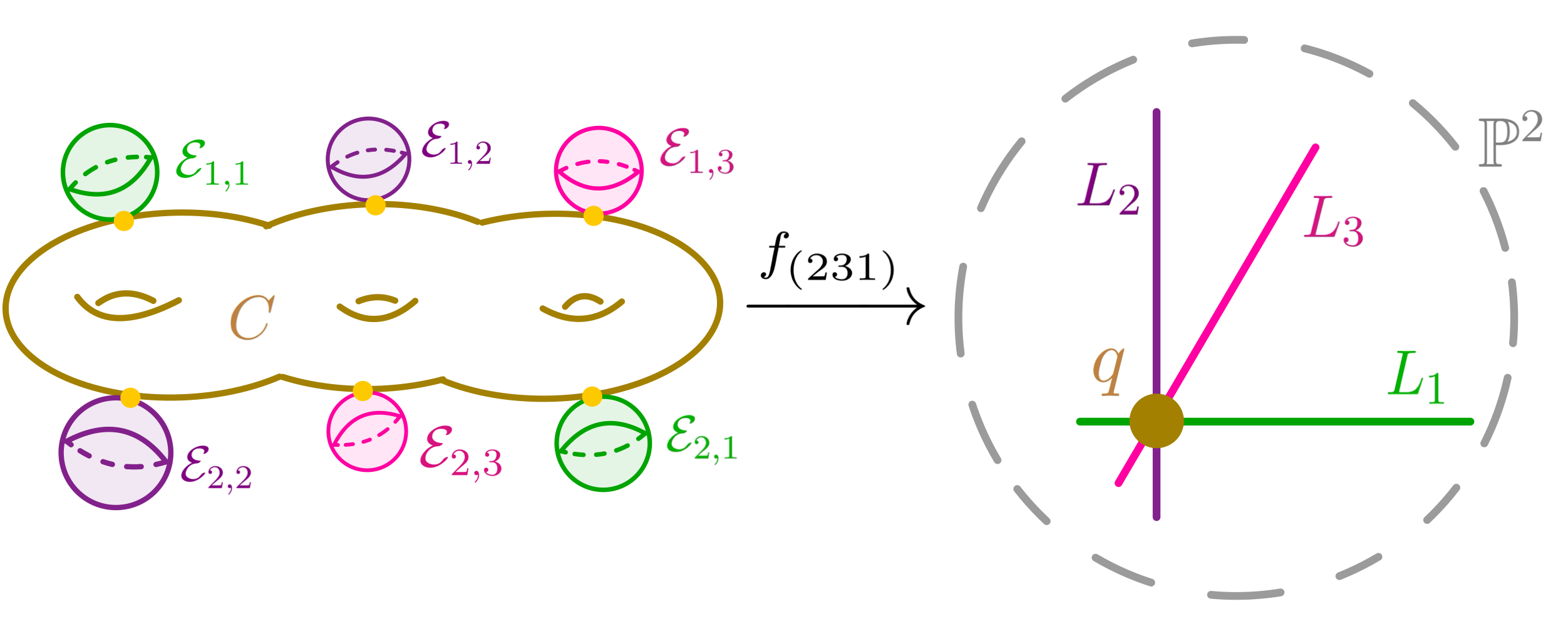}
    \caption{An example of a non-eventually smoothable stable map $f_{\sigma}$ for $h = 3$ and $\sigma=(231)\in\fS_3$ as in Lemma \ref{Lem: sigmaNoFixedPoint}. 
    }
    \label{fig:h=3,N=2 nonsmoothable}
\end{figure}

\begin{Lemma}\label{Lem: sigmaNoFixedPoint}
    If $\sigma$ has no fixed points, then $f_\sigma:\sC_\sigma\to\bP^2$ is not eventually smoothable.
\end{Lemma}
\begin{proof}
    For each $1\le j\le h$, note that
    \begin{enumerate}[(i)]
        \item $\delta_{C,p_{1,j}}$ and $\delta_{C,p_{2,j}}$ have the same line $\ell_j\subset H^1(C,\clO_C)$ as their image, and
        \item we have $T_{\bP^2,q} = T_{L_j,q}\oplus T_{L_{\sigma(j),q}}$ since $j\ne\sigma(j)$.
    \end{enumerate}
    Thus, the map \eqref{eqn:first-order-obstruction-map} is an isomorphism. The claim now follows from Theorem \ref{thm:first-order-obstruction}.    
\end{proof}

\begin{Lemma} \label{Lem: sigma=ID}
    If $\sigma = \normalfont\text{id}\in\fS_h$, then $f_\sigma:\sC_\sigma\to\bP^2$ is smoothable.
\end{Lemma}
\begin{proof}
    Write $\sC = \sC_\text{id}$ and $f = f_\text{id}$. Factor the canonical morphism \eqref{eqn:canonical-morphism} through a double cover $\varphi:C\to D$ where $D\subset\bP H^0(C,\omega_C)^\vee$ is a rational normal curve. 
    
    For $1\le j\le h$, let $p_j = \varphi(p_{1,j}) = \varphi(p_{2,j})\in D$ and let $(\sE_j,p'_j)$ be a copy of $(L_j,q)$. Let $\sD$ be the prestable curve obtained from $D\sqcup\bigsqcup\sE_j$ by making the identifications $p_j\sim p_j'$. We have a natural map $\psi:\sC\to\sD$ given by mapping $C$ to $D$ via $\varphi$ and mapping $\sE_{i,j}$ to $\sE_j$ by the identity map for $1\le i\le 2$ and $1\le j\le h$. It is clear that we have a factorization $f = g\circ\psi$ for a morphism $g:\sD\to\bP^2$. See Figure \ref{fig:h=3,N=2 smoothable}.

    \begin{figure}[h]
    \centering
    \includegraphics[width=0.95\linewidth]{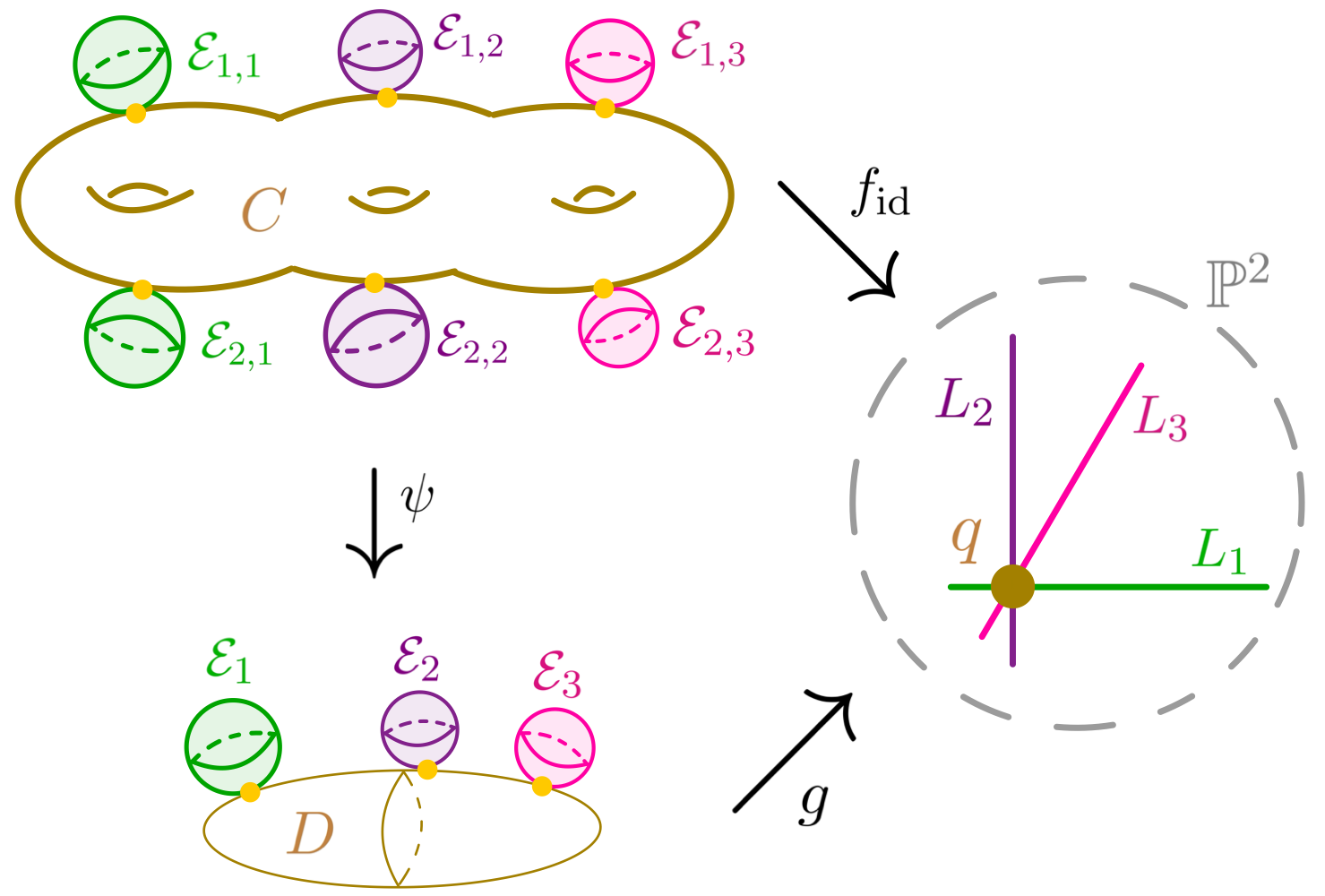}
    \caption{The smoothable stable map $f_{\text{id}} = g\circ\psi$ from Lemma \ref{Lem: sigma=ID} in the case where $h=3$. 
    }
    \label{fig:h=3,N=2 smoothable}
    \end{figure}

    Define $\pi_\sY:\sY\to\bA^1$ by blowing up $D\times\bA^1$ at the points $(p_j,0)$ for $1\le j\le h$. Similarly, define $\pi_\sX:\sX\to\bA^1$ by blowing up $C\times\bA^1$ at the points $(p_{i,j},0)$ for $1\le i\le 2$ and $1\le j\le h$. Then, $\varphi\times\text{id}_{\bA^1}$ induces a morphism $\Psi:\sX\to\sY$ of families over $\bA^1$ whose restriction $\pi_\sX^{-1}(0)\to\pi_\sY^{-1}(0)$ is identified with $\psi:\sC\to\sD$.
    
    Pick a deformation of $g:\sD\to\bP^2$ to the generic fibre of the family $\pi_\sY:\sY\to\bA^1$. This is possible since the relevant obstructions lie in $H^1(\sD,g^*T_{\bP^2}) = 0$.\footnote{The obstruction vanishes due to the Euler exact sequence $0\to\clO_{\bP^2}\to\clO_{\bP^2}(1)^{\oplus 3}\to T_{\bP^2}\to 0$ and the fact that the pullback of $\clO_{\bP^2}(1)$ has vanishing $H^1$ on the genus $0$ curve $\sD$.} Composing this deformation with $\Psi:\sX\to\sY$ yields a smoothing of $f = g\circ\psi:\sC\to\bP^2$.
\end{proof}

\begin{Remark}
    The stable maps $f_\sigma:\sC_\sigma\to\bP^2$ for different choices of $\sigma\in\fS_h$ have the same image and are obtained from one another by permuting the components of the effective sub-curve. Thus, a purely numerical obstruction to smoothing (such as the one in \cite{Vakil-stable-red} based on $\delta$-invariants) cannot distinguish between them.
\end{Remark}

\section{Proof of the main result}\label{sec:proof}
In this section, we prove Theorem~\ref{thm:first-order-obstruction}. Consider a non-constant stable map $f:\sC\to X$, with effective sub-curve $\sE$ and a ghost component $C$ mapping to $q\in X$. Assuming that $f:\sC\to X$ is eventually smoothable, we will produce a nonzero element in the kernel of the linear map \eqref{eqn:first-order-obstruction-map}. We divide the argument into steps for clarity. 

\subsection{Reduction from eventual smoothability to smoothability}

As $f:\sC\to X$ is eventually smoothable, there is an embedding $\iota:X\hookrightarrow X'$ into another smooth projective variety such that $\iota\circ f:\sC\to X'$ is smoothable. The answer to whether or not the map \eqref{eqn:first-order-obstruction-map} is injective remains unchanged when we replace $f:\sC\to X$ by $\iota\circ f:\sC\to X'$. Therefore, replacing $X$ by $X'$, we assume for the remainder that $f:\sC\to X$ is smoothable.

\subsection{Smoothing family and its resolution}\label{subsec:Smoothing family and its resolution}
Since $f:\sC\to X$ is smoothable, we may choose a smooth affine algebraic curve $\Delta$ with a closed point $s\in\Delta$, a flat family $\pi:\sS\to\Delta$ of prestable curves and a morphism $F:\sS\to X$ with the following properties.
\begin{enumerate}[(a)]
    \item The restriction of $F:\sS\to X$ over $s\in\Delta$ is identified with $f:\sC\to X$.
    \item The restriction of $\pi$ over $\Delta\setminus\{s\}$ is a smooth morphism $\sS\setminus\sC\to\Delta\setminus\{s\}$.
    \item There is a regular function $t$ on $\Delta$ which vanishes precisely at $s$ and restricts to a local coordinate for $\Delta$ at $s$. Thus, we may regard $\Gamma(\Delta,\clO_{\Delta})$ as an \'etale $k[t]$-algebra.
\end{enumerate}

Write $C\cap\sE = \{p_1,\ldots,p_n\}$. For $1\le i\le n$, using the \'etale local description of a family of nodal curves \cite[\href{https://stacks.math.columbia.edu/tag/0CBY}{Tag~0CBY}]{stacks-project}, we can find 
\begin{enumerate}[(i)]
    \item an affine \'etale neighborhood $\Spec B_i\to \sS$ of $p_i$,
    \item an affine \'etale neighborhood $\Spec A_i\to\Delta$ of $s$,
    \item an integer $m_i\ge 1$, and an \'etale map 
    \begin{align*}
        \frac{A_i[x_i,y_i]}{(x_iy_i-t^{m_i})}\to B_i   
    \end{align*}
    of $k$-algebras,
\end{enumerate}
such that the induced map $A_i\to B_i$ (or equivalently, $\Spec B_i\to\Spec A_i$) is compatible with $\pi$. Here, $x_i$ and $y_i$ are variables generating the free polynomial $A_i$-algebra $A_i[x_i,y_i]$. 

Therefore, we have a commutative diagram
\begin{equation}\label{eqn:local-model-interface-nodes}
\begin{tikzcd}
    \sS \arrow[d,"\pi"] & \arrow[l,"\text{\'etale}"'] \Spec B_i \arrow[d] \arrow[r,"\text{\'etale}"] &\Spec \displaystyle\frac{k[x_i,y_i,t]}{(x_iy_i-t^{m_i})} \arrow[d] \\
    \Delta & \arrow[l,"\text{\'etale}"'] \Spec A_i \arrow[r,"\text{\'etale}"] & \Spec k[t]
\end{tikzcd}
\end{equation}
where the horizontal arrows are \'etale. From the description of the versal deformation of a node \cite[Theorem~14.1]{Har-DT}, we find that $m_i$ is intrinsically characterized as the largest integer for which the restriction of $\sS\to\Delta$ to $\Spec({\clO_{\Delta,s}}/{(t^{m_i})})$ induces a trivial deformation of the node at $p_i$. By switching the variables if needed, we assume that the image of $x_i,y_i$ in $B_i/{(t)}$ corresponds to local coordinates at $p_i$ on $\sE,C$, respectively.

We have a description similar to \eqref{eqn:local-model-interface-nodes} at all the other nodes of $\sC$ as well. Thus, the surface $\sS$ is non-singular except at those nodes of $\sC$ where $\sS$ is modeled on the $A_l$-singularity $xy=t^{l+1}$ for some integer $l\ge 1$. We now recall a basic fact, illustrated in Figure~\ref{fig:resol-Al-sing}.
\begin{Lemma}[Resolution of $A_l$-singularity]\label{lemma:resol-Al-sing}
    Let $l\ge 1$ be an integer. Repeatedly blowing up the singular points of the surface
    \begin{align*}
        \Spec\frac{ k[x,y,t]}{(xy-t^{l+1})}    
    \end{align*}
    results in a non-singular surface, which can be described globally as the vanishing locus
    \begin{align}\label{eqn:resol-Al-sing}
        V(xu_1-tv_1,v_1u_2-tv_2u_1,\ldots,v_{l-1}u_l-tv_lu_{l-1},u_lt-v_ly)
    \end{align}
    inside $\Spec k[x,y,t]\times\prod_{j=1}^l\Proj k[u_j,v_j]$.
    In particular, the exceptional locus above the point $(x,y,t)$ is a chain of $l$ copies of $\bP^1$, each of self-intersection $-2$. One end of the chain meets the strict transform of the curve defined by $(x,t)$ while the other end meets the strict transform of the curve defined by $(y,t)$.
\end{Lemma}

\begin{figure}[h]
\centering
\includegraphics[width=12cm]{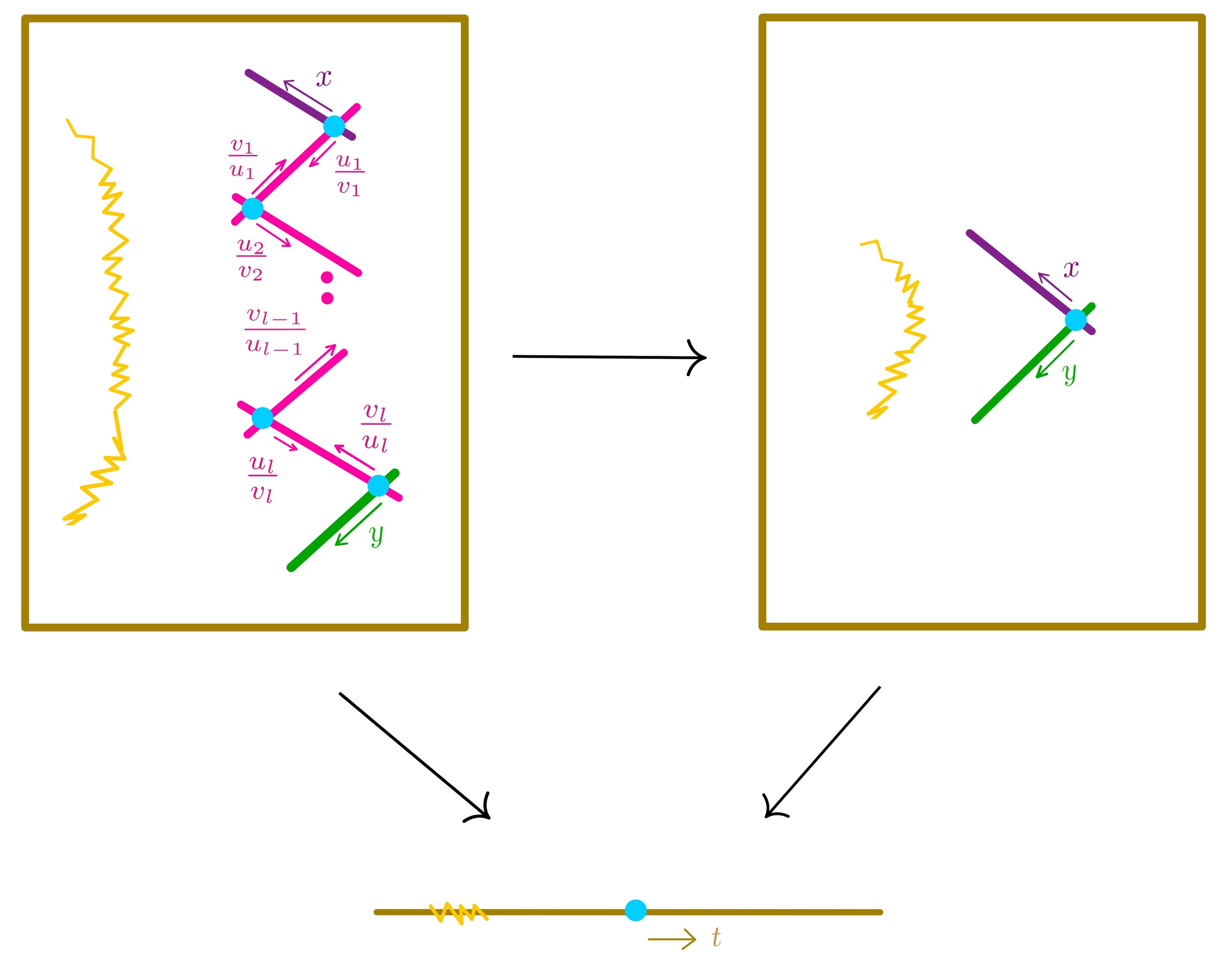}
\caption{Resolving the $A_l$-singularity $xy=t^{l+1}$ by repeatedly blowing up singular points. At each of the $l+1$ nodes of the $t = 0$ fibre of the resolved surface, the labels specify a pair of local coordinates whose product is $t$.}
\label{fig:resol-Al-sing}
\end{figure}

\begin{proof}
    This follows from a direct computation.
    %
\end{proof}

Since blow-ups can be computed \'etale locally \cite[\href{https://stacks.math.columbia.edu/tag/085S}{Tag~085S}]{stacks-project}, the local models \eqref{eqn:local-model-interface-nodes} and Lemma~\ref{lemma:resol-Al-sing} allow us to resolve the singularities of $\sS$ by repeatedly blowing them up.
Let $\rho:\tilde\sS\to\sS$ be the resolution of the surface $\sS$ obtained by this procedure; see Figure~\ref{fig:resol-smoothing-family}.
For $1\le i\le n$, let $\tilde\sS_{B_i}$ be the resolution of $\Spec B_i$ obtained by repeatedly blowing up the unique singular point.
The diagram \eqref{eqn:local-model-interface-nodes} then extends to a commutative diagram
    \begin{equation}\label{eqn:local-model-resol-interface-nodes}
    \begin{tikzcd}
        \tilde\sS \arrow[d,"\rho"] & \arrow[l,"\text{\'etale}"'] \tilde\sS_{B_i} \arrow[dl,phantom,"\llcorner",pos=0.1] \arrow[dr,phantom,"\lrcorner",pos=0.08] \arrow[rr,"\text{\'etale}"] \arrow[d] & & V(x_iu_{i,1}-tv_{i,1},\ldots,u_{i,m_i-1}t-v_{i,m_i-1}y_i) \arrow[d] \\
        \sS \arrow[d,"\pi"] & \arrow[l,"\text{\'etale}"'] \Spec B_i \arrow[d] \arrow[rr,"\text{\'etale}"] & \phantom{} & \displaystyle\Spec\frac{ {k[x_i,y_i,t]}}{{(x_iy_i-t^{m_i})}} \arrow[d] \\
    \Delta & \arrow[l,"\text{\'etale}"'] \Spec A_i \arrow[rr,"\text{\'etale}"] & & \Spec k[t]
    \end{tikzcd}
    \end{equation}
    where the horizontal arrows are \'etale, the top two squares are pullbacks and the top right position has the subscheme of $\Spec k[x_i,y_i,t]\times\prod_{j=1}^{m_i-1}\Proj k[u_{i,j},v_{i,j}]$ analogous to \eqref{eqn:resol-Al-sing}.

\begin{figure}[h]
\centering
\includegraphics[width=11cm]{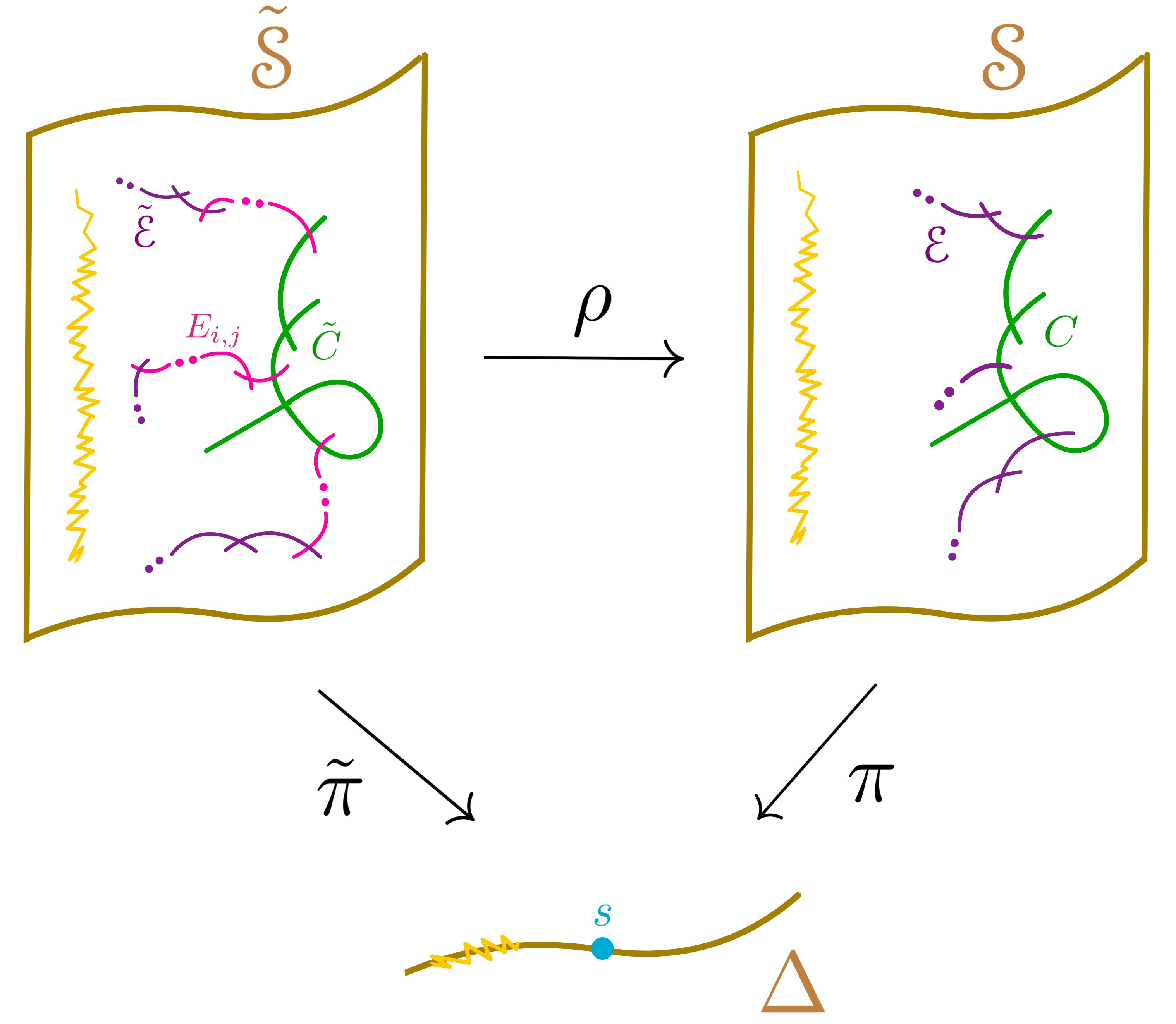}
\caption{Resolving the singularities of $\sS$ by repeatedly blowing up the singular points to get $\tilde\sS$, with exceptional curves $E_{i,j}$.}
\label{fig:resol-smoothing-family}
\end{figure}

\begin{notation}[Coordinates on $\tilde\sS$]\label{nota:coordinates}    
    Define $\tilde\pi:\tilde \sS\to\Delta$ and $\tilde F:\tilde\sS\to X$ to be the morphisms induced by the morphisms $\pi$ and $\rho$ respectively.
    Write $\tilde\sC = \tilde\pi^{-1}(s)$. Define
    \begin{align*}
        \tilde C = \overline{\rho^{-1}(C\setminus\{p_1,\ldots,p_n\})}\quad\text{and}\quad\tilde\sE = \overline{\rho^{-1}(\sE\setminus\{p_1,\ldots,p_n\})}.
    \end{align*}
    Sequentially number the copies of $\bP^1$ in the exceptional locus over $p_i$ as $E_{i,1},\ldots,E_{i,m_i-1}$, with the convention that $E_{i,1}$ meets $\tilde\sE$ while $E_{i,m_i-1}$ meets $\tilde C$. 
    For $1\le i\le n$ and $1\le j\le m_i$, write $\{p_{i,j}\} = E_{i,j-1}\cap E_{i,j}$ (see Figure~\ref{fig:Ghost}), where we set $E_{i,0} = \tilde\sE$ and $E_{i,m_i} = \tilde C$ for convenience. For the purpose of coordinate computations, on each $\tilde\sS_{B_i}$, use \eqref{eqn:local-model-resol-interface-nodes} to define the rational functions $x_{i,1} = x_i$, $y_{i,m_i} = y_i$, and
    \begin{align*}
        x_{i,j+1} = \frac{v_{i,j}}{u_{i,j}}\quad\text{and}\quad y_{i,j} = \frac{u_{i,j}}{v_{i,j}}, 
    \end{align*}
    for $1\le j\le m_i-1$. Note that $x_{i,j}$ and $y_{i,j}$ provide local coordinates at (the inverse image of) $p_{i,j}$ for (the inverse images of) $E_{i,j-1}$ and $E_{i,j}$ respectively and satisfy the relation $x_{i,j}y_{i,j} = t$ on $\tilde\sS_{B_i}$.
\end{notation}

The Cartier divisor $s$ on $\Delta$ pulls back under $\tilde\pi$ to the Cartier divisor $\tilde\sC$ on $\tilde\sS$ and so, we have a canonical identification of the associated locally free sheaves
\begin{align}\label{eqn:pullback-line-bundle}
    \tilde\pi^*(\clO_\Delta(s)) \xrightarrow{\simeq} \clO_{\tilde\sS}(\tilde\sC).
\end{align}
The following elementary observation will be very useful.

\begin{Lemma}\label{lemma:smoothing-parameter-iso}
    The family $\pi:\sS\to\Delta$ determines a canonical linear isomorphism
    \begin{align}\label{eqn:smoothing-parameter-iso}
        \Phi_{p_i}:T_{\Delta,s}^{\otimes m_i}\xrightarrow{\simeq} T_{\sE,p_i}\otimes T_{C,p_i},
    \end{align}
    which is given by $\partial_t^{\otimes m_i}\mapsto\partial_{x_i}\otimes\partial_{y_i}$ in any choice of local model as in \eqref{eqn:local-model-interface-nodes}.
\end{Lemma}
\begin{proof}
    We first consider the special case where $m_i = 1$. Then, $\sS$ is smooth at $p_i$ and so, we have $\tilde\sS = \sS$ and $\tilde\sC = \sC$ in a neighborhood of $p_i$. Clearly, $\pi^*(\clO_\Delta(s))|_{p_i} = T_{\Delta,s}$. The divisor $\sC\subset\sS$ is given by the sum of the divisors $\sE$ and $C$ in a neighborhood of $p_i$ and so,
    \begin{align*}
        \clO_{\sS}(\sC)|_{p_i} = \clO_{\sS}(C+\sE)|_{p_i} = \clO_{\sS}(C)|_{p_i}\otimes\clO_{\sS}(\sE)|_{p_i} \xrightarrow{\simeq} T_{\sE,p_i}\otimes T_{C,p_i}.
    \end{align*}
    Here, we have used $\clO_{\sS}(C)|_C = N_{C/\sS}$ and the isomorphism $N_{C/\sS}|_{p_i} \simeq T_{\sE,p_i}$ which follows from the fact that $\sE,C$ meet transversally at $p_i$ and the analogous facts for $\clO_{\sS}(\sE)|_\sE$ and $N_{\sE/\sS}|_{p_i}$. This computation, together with the restriction of \eqref{eqn:pullback-line-bundle} to $p_i$, defines $\Phi_{p_i}$.

    We now treat the general case. As in the previous paragraph, we obtain isomorphisms
    \begin{align}\label{eqn:intermediate-smoothing-parameter-iso}
        \Phi_{p_{i,j}}:T_{\Delta,s}\xrightarrow{\simeq} T_{E_{i,j-1},p_{i,j}}\otimes T_{E_{i,j},p_{i,j}}
    \end{align}
    by restricting \eqref{eqn:pullback-line-bundle} to $p_{i,j}$ for $1\le j\le m_i$. The tensor product $\bigotimes_{1\le j\le m_i}\Phi_{p_{i,j}}$, together with Lemma~\ref{lemma:P1-tgt-space} below, defines \eqref{eqn:smoothing-parameter-iso}. We check, using \eqref{eqn:local-model-resol-interface-nodes}, that $\Phi_{p_{i,j}}$ is given in coordinates by $\partial_t\mapsto\partial_{x_{i,j}}\otimes\partial_{y_{i,j}}$, which implies that $\Phi_{p_i}$ is given in coordinates by $\partial_t^{\otimes m_i}\mapsto\partial_{x_i}\otimes\partial_{y_i}$.
\end{proof}

\begin{Lemma}\label{lemma:P1-tgt-space}
    Let $E$ be a non-singular projective curve of genus $0$, and let $p,q\in E$ be two distinct closed points. Then the line bundle $\sL = \clO_E(p-q)$ is trivial. Moreover, evaluating the global sections of $\sL$ at $p$ and $q$ yields isomorphisms $H^0(E,\sL)\simeq\sL|_p = T_{E,p}$ and $H^0(E,\sL)\simeq\sL|_q = T_{E,q}^\vee$. In particular, there is a canonical isomorphism $T_{E,p}\simeq T_{E,q}^\vee$.
\end{Lemma}
\begin{proof}
    This is clear once we choose an isomorphism $E\simeq\bP^1$. 
\end{proof}

\subsection{Rational function from smoothing family}

After suitably renumbering the $p_i$, we assume that we have $m := m_1 = \ldots = m_r < m_{r+1}\le\cdots\le m_n$ for some integer $1\le r\le n$. The next proposition is the crux of the proof of Theorem~\ref{thm:first-order-obstruction}.

\begin{Proposition}[Leading term of smoothing family]\label{prop:simple-pole-function}
   Consider the locally free sheaf
    \begin{align*}
        \fF = \clO_C(p_1 + \cdots + p_r)\otimes \Hom(T_{\Delta,s}^{\otimes m},T_{X,q}).
    \end{align*}
    It has a section $\sigma\in H^0(C,\fF)$ with the following property. For $1\le i\le r$, we have
    \begin{align*}
        \sigma(p_i) = \left(\normalfont\text{id}_{T_{C,p_i}}\otimes d(f|_\sE)_{p_i}\right) \circ \Phi_{p_i},
    \end{align*}
    under the identification $\fF|_{p_i} = T_{C,p_i}\otimes \Hom(T_{\Delta,s}^{\otimes m},T_{X,q})= \Hom(T_{\Delta,s}^{\otimes m},T_{C,p_i}\otimes T_{X,q})$.
\end{Proposition}

The proof of Proposition~\ref{prop:simple-pole-function} is discussed in the next subsection. The following corollary shows how to deduce Theorem~\ref{thm:first-order-obstruction} from this proposition.

\begin{Corollary}
    The composite linear map
    \begin{align*}
        T_{\Delta,s}^{\otimes m}\xrightarrow{(\Phi_{p_1},\ldots,\Phi_{p_r})}\bigoplus_{1\le i\le r}T_{C,p_i}\otimes T_{\sE,p_i}\hookrightarrow\bigoplus_{1\le i\le n}T_{C,p_i}\otimes T_{\sE,p_i}\xrightarrow{\eqref{eqn:first-order-obstruction-map}} H^1(C,\clO_C)\otimes T_{X,q}
    \end{align*}
    is zero. In particular, the linear map \eqref{eqn:first-order-obstruction-map} has a non-trivial kernel.
\end{Corollary}
\begin{proof}
    Write $W := \Hom(T_{\Delta,s}^{\otimes m},T_{X,q})$, and consider the short exact sequence
    \begin{align*}
        0\to\fF(-p_1-\cdots-p_r)\to\fF\to\bigoplus_{1\le i\le r}\fF|_{p_i}\to 0.
    \end{align*}
    The corresponding cohomology exact sequence contains
    \begin{align}\label{eqn:cohomology-les}
        H^0(C,\fF)\to\bigoplus_{1\le i\le r} T_{C,p_i}\otimes W\to H^1(C,\clO_C)\otimes W. 
    \end{align}
    The first map in \eqref{eqn:cohomology-les} is given by evaluating sections at $p_1,\ldots,p_r$, while the second map is given by $\bigoplus_{1\le i\le r}\delta_{C,p_i}\otimes\text{id}_W$.
    Let $\sigma\in H^0(C,\fF)$ be as in Proposition~\ref{prop:simple-pole-function}. 
    
    From \eqref{eqn:cohomology-les}, we deduce that the image of $\sigma$ in $H^1(C,\clO_C)\otimes W$ vanishes. The desired claim now follows from the explicit expressions for $\sigma(p_1),\ldots,\sigma(p_r)$ given in Proposition~\ref{prop:simple-pole-function}.
\end{proof}

\subsection{Key computation in coordinates}

The proof of Proposition~\ref{prop:simple-pole-function} uses an explicit coordinate computation carried out in Lemma~\ref{lemma:main-coord-computation} below.

\begin{notation}[Ghost curves in $\tilde\sS$]\label{nota:ghost-curves-in-S}
    For each integer $1\le l\le m$, define the curve
    \begin{align*}
        \tilde C_l = \tilde C\cup \bigcup_{i=1}^n \bigcup_{j=l}^{m_i-1} E_{i,j},
    \end{align*}
    and note that $p_{1,l},\ldots,p_{n,l}\in\tilde C_l$ (see Figure~\ref{fig:Ghost}). Moreover, we have 
    \begin{align*}
        \tilde C\cap\{p_{1,m},\ldots,p_{n,m}\} = \{p_{1,m},\ldots,p_{r,m}\}.
    \end{align*}
\end{notation}

\begin{notation}[Local coordinates on $X$ near $q$]
    Lift the identity map of $T_{X,q}^\vee$, which we regard as an element of
    \begin{align*}
        \Hom(T_{X,q}^\vee,T_{X,q}^\vee) = T_{X,q}\otimes T_{X,q}^\vee = T_{X,q}\otimes(\fm_{X,q}/\fm_{X.q}^2),    
    \end{align*}
    to an element of $T_{X,q}\otimes\fm_{X,q}$. This yields a rational function $\varphi:X\dashrightarrow T_{X,q}$. Note that $\varphi$ defines a regular function on an open neighborhood $U$ of $q$, and satisfies $\varphi(q) = 0$ and $d\varphi_q = \text{id}_{T_{X,q}}$. From this, we get the rational function
    \begin{align*}
        G = \varphi\circ \tilde F:\tilde\sS\dashrightarrow T_{X,q},    
    \end{align*}
    which defines regular function on the open neighborhood $\tilde F^{-1}(U)$ of $\tilde C_1$.
\end{notation}

\begin{figure}[h]
\centering
\includegraphics[width=5.4cm]{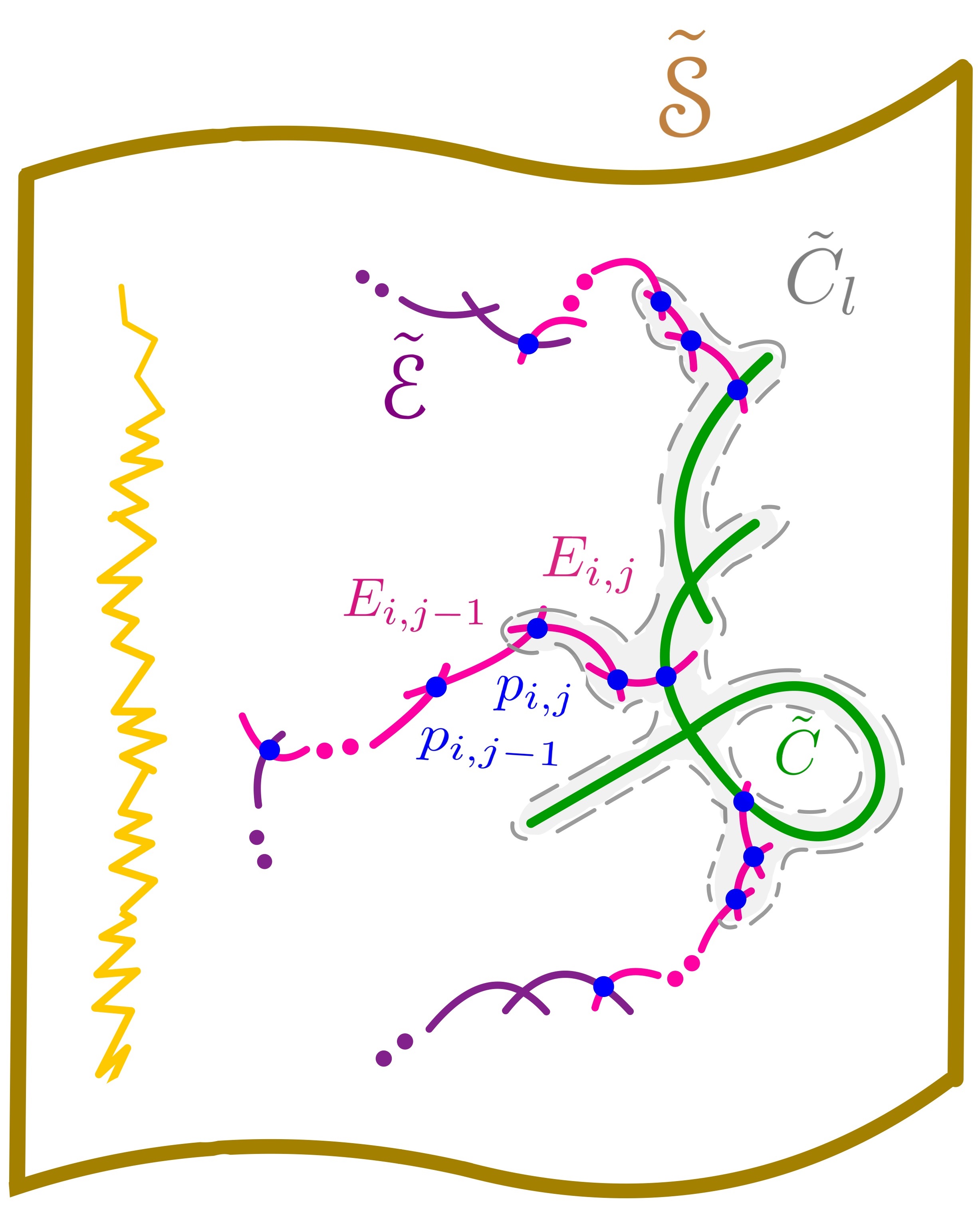}
\caption{A ghost curve $\tilde C_{l}$ in $\tilde \sS$ (Notation~\ref{nota:ghost-curves-in-S}) which is depicted as the union of the curves in the gray region. In this figure, we have $l=m-2$.}
\label{fig:Ghost}
\end{figure}

\begin{Lemma}\label{lemma:main-coord-computation}
    For each integer $1\le l\le m$, we have a rational function $G_l:\tilde\sS\dashrightarrow T_{X,q}$ and constant vectors $a_0,\ldots,a_{l-1}\in T_{X,q}$ such that
   \begin{align*}
        G = a_0 + a_1t + \cdots + a_{l-1}t^{l-1} + t^lG_l.
    \end{align*}
   The restriction of $G_l$ to $\tilde C_l$ is regular, except possibly at $p_{1,l},\ldots,p_{n,l}$. For $1\le i\le n$, the function $G_l|_{\tilde C_l}$ has at worst a simple pole at $p_{i,l}$, with residue given by
    \begin{align*}
        \partial_{y_{i,l}}\otimes\left.\frac{\partial G}{\partial {x_{i,1}}}\right\vert_{p_{i,1}}\in T_{E_{i,l},p_{i,l}}\otimes T_{X,q}.
    \end{align*}
\end{Lemma}
\begin{proof}
    The proof is by induction on $l$. For \'etale local calculations using the coordinate description \eqref{eqn:local-model-resol-interface-nodes} and Notation~\ref{nota:coordinates}, we conflate $\tilde\sS_{B_i}$ with $\tilde\sS$ to improve readability.
    
    For the base case, take $a_0 = 0$ and define $G_1 = t^{-1}G$. As $F$ maps $\tilde C_1$ to $q$, the function $G$ identically vanishes on $\tilde C_1$. Since $\tilde C_1\setminus\{p_{i,1}\}_{i=1}^n\subset\tilde\sS\setminus\tilde\sE$ is defined by the equation $t = 0$, it follows that $G_1$ restricts to a regular function on $\tilde C_1\setminus\{p_{i,1}\}_{i=1}^n$. To examine the possible pole at $p_{1,i}$, we work \'etale locally using \eqref{eqn:local-model-resol-interface-nodes}. The curve $\tilde C_1\subset\tilde\sS$ is defined near $p_{1,i}$ by the equation $x_{i,1} = 0$. Thus, we have $G = x_{i,1}h$ for some $h\in\clO_{\sS,p_{i,1}}$ with 
    \begin{align*}
        h(p_{i,1}) = \left.\frac{\partial G}{\partial x_{i,1}}\right\vert_{p_{i,1}}.  
    \end{align*}
    Using $x_{i,1}y_{i,1} = t$, we get $G_1 = t^{-1}G = y_{i,1}^{-1}h$, which concludes the base case.
    
    For the induction step, let $2\le l\le m$ be given and assume that we have already found $a_0,\ldots,a_{l-2}\in T_{X,q}$ and $G_{l-1}$ as in the lemma. By assumption, the $G_{l-1}$ is regular and therefore constant on the connected curve $\tilde C_l$. Let $a_{l-1}\in T_{X,q}$ be this constant value and define $G_l = t^{-1}(G_{l-1} - a_{l-1})$. Arguing as in the previous paragraph, we see that the restriction of $G_l$ to $\tilde C_l$ is regular except for (at worst) simple poles at the points $p_{i,l}$, for $1\le i\le n$, with residues given by
    \begin{align*}
        \partial_{y_{i,l}}\otimes\left.\frac{\partial(G_{l-1}-a_{l-1})}{\partial x_{i,l}}\right\vert_{p_{i,l}}\in T_{E_{i,l},p_{i,l}}\otimes T_{X,q}.
    \end{align*}
    The restriction of $G_{l-1}-a_{l-1}$ to the non-singular genus $0$ curve $E_{i,l-1}$ defines a section of $\clO_{E_{i,l-1}}(p_{i,l-1}-p_{i,l})\otimes T_{X,q}$.  By the inductive hypothesis, $(G_{l-1}-a_{l-1})|_{E_{i,l-1}}$ has residue
    \begin{align*}
        \partial_{y_{i,l-1}}\otimes\left.\frac{\partial G}{\partial x_{i,1}}\right\vert_{p_{i,1}}\in T_{E_{i,l-1},p_{i,l-1}}\otimes T_{X,q}
    \end{align*}
    at $p_{i,l-1}$. Note that the isomorphism in Lemma~\ref{lemma:P1-tgt-space} for $E_{i,l-1}$ identifies the derivative of $(G_{l-1}-a_{l-1})|_{E_{i,l-1}}$ at $p_{i,l}$ with the residue of $(G_{l-1}-a_{l-1})|_{E_{i,l-1}}$ at $p_{i,l-1}$. Since this isomorphism maps $dx_{i,l}\in T_{E_{i,l-1},p_{i,l}}^\vee$ to $\partial_{y_{i,l-1}}\in T_{E_{i,l-1},p_{i,l-1}} 
    $, we get
    \begin{align*}
        \left.\frac{\partial(G_{l-1}-a_{l-1})}{\partial x_{i,l}}\right\vert_{p_{i,l}} = \left.\frac{\partial G}{\partial x_{i,1}}\right\vert_{p_{i,1}}.
    \end{align*}
    This concludes the induction step.
\end{proof}

Finally, we show how Proposition~\ref{prop:simple-pole-function} follows from Lemma~\ref{lemma:main-coord-computation}.

\begin{proof}[Proof of Proposition~\ref{prop:simple-pole-function}]
    Restrict the function $G_m$ from Lemma~\ref{lemma:main-coord-computation} to the curve $\tilde C\subset\tilde C_m$ to obtain a section of $\clO_{\tilde C}(p_{1,m} + \cdots + p_{r,m})\otimes T_{X,q}$. The morphism $\tilde C\to C$ contracts copies of $\bP^1$ on which $G_m$ restricts to a regular and therefore constant function with values in $T_{X,q}$. Thus, $G_m$ descends to a section of $\clO_C(p_1+\cdots+p_r)\otimes T_{X,q}$, which we denote as $G_m|_C$. 
    
    We can now define $\sigma$ by the formula $\sigma(\partial_t^{\otimes m}) = G_m|_C$. The $l = m$ case of Lemma~\ref{lemma:main-coord-computation} specifies the value of $\sigma(p_i)$, for $1\le i\le r$. Recalling the identities
    \begin{align*}
        x_{i,1} = x_i,\quad y_{i,m} = y_i,\quad \Phi_{p_i}(\partial_t^{\otimes m}) = \partial_{x_i}\otimes\partial_{y_i},\quad \left.\frac{\partial G}{\partial x_{i,1}}\right\vert_{p_{i,1}} = d(f|_\sE)_{p_i}(\partial_{x_i})
    \end{align*}
    gives the desired conclusion.
\end{proof}

\bibliographystyle{plain}
\bibliography{references}

\end{document}